\documentclass[12pt]{amsart}
\usepackage{graphicx, nccmath}

\usepackage{amsfonts, amssymb, amsthm, amsmath, braket, hyperref, mathtools, mathrsfs}
\allowdisplaybreaks

\newtheorem*{introtheorem}{Theorem}
\newtheorem{theorem}{Theorem}[section]

\newtheorem{corollary}[theorem]{Corollary}
\newtheorem{cor}[theorem]{Corollary}
\newtheorem{lemma}[theorem]{Lemma}

\newtheorem{example}[theorem]{Example}
\newtheorem{ex}[theorem]{Example}

\theoremstyle{definition}

\newtheorem{remark}[theorem]{Remark}

\newcommand{\defi}[1]{\textsf{#1}} 

\newcommand{\bbC}{\mathbb{C}}

\newcommand{\bbN}{\mathbb{N}}

\newcommand{\bbR}{\mathbb{R}}

\newcommand{\bbZ}{\mathbb{Z}}

\newcommand{\prm}{^\prime}
\newcommand{\parent}[1]{\left( #1 \right)}

\newcommand{\abs}[1]{\left|#1\right|}

\newcommand{\floor}[1]{\left\lfloor #1 \right\rfloor}

\newcommand{\xn}[1]{X_{#1}}
\newcommand{\xmn}[2]{\parent{X^{#1}}_{#2}}

\newcommand{\amn}[2]{\parent{A_{#1}}_{#2}}

\newcommand{\axn}[1]{A(X)_{#1}}
\newcommand{\apxn}[1]{A\prm(X)_{#1}}

\newcommand{\derrorn}[1]{\delta_{#1}}
\newcommand{\eerrorn}[1]{\epsilon_{#1}}

\begin{document}
\title{Fast-growing series are transcendental}
\author{Robert J. MacG. Dawson}\email{rdawson@cs.smu.ca}\address{Dept. of Mathematics and Computing Science\\Saint Mary's University, Halifax, NS, Canada\\B3H 3C3}
\author{Grant Molnar}\email{Grant.S.Molnar.GR@dartmouth.edu}\address{Mathematics Department\\Dartmouth College, Hanover, NH, USA\\245 Kemeny Hall}

\begin{abstract}
Let $R$ be a subring of $\bbC[[z]]$, and let $X \in \bbC[[z]]$. The Newton-Puiseux Theorem implies that if the coefficients of $X$ grow sufficiently rapidly relative to the coefficients of the series in $R$, then $X$ is transcendental over $R$. We prove an alternative proof of this result by establishing a relationship between the coefficients of $A(X)$ and $A^\prime(X)$, where $A(T)$ is a polynomial over $\bbC[[z]]$. 
\end{abstract}

\maketitle

\section{Introduction}

For $C(z) = C$ a power series in $\bbC[[z]]$, we write $\parent{C}_n$ for the $n$th coefficient of $C$, or $C_n$ if no confusion arises. Throughout this paper, indices of series and sequences are always nonnegative unless otherwise noted.

Suppose we are given a ring $R \subseteq \bbC[[z]]$ and a power series $X \in \bbC[[z]]$. It is natural to ask what relationship $R$ bears to $X$. For instance, we may ask whether $X$ is algebraic or transcendental over $R$. (Recall that $X$ is \defi{algebraic} over $R$ if there is a nonzero polynomial $F(T) \in R[T]$ such that $F(X) = 0$, and $X$ is \defi{transcendental} over $R$ otherwise.)

Transcendental elements are useful building blocks in the theory of commutative rings. Indeed, $X$ is transcendental over $R$ precisely if the ring $R[X]$ satisfies the following universal property: for every $R$-algebra $S$, and every element $x \in S$, there is a unique $R$-algebra morphism $f : R[X] \to S$ such that $f(X) = x$. But if $X$ is not transcendental over $R$, our options for $f(X)$ are sharply curtailed. For instance, if $X \not\in R$ but $X^2 \in R$, then an $R$-algebra morphism $f : R[X] \to S$ must map $X$ to a square root of $X^2$ in $S$; any other series algebraic over $R$ is subject to similar restrictions.

However, transcendental elements are slippery. Indeed, many decades elapsed between Euler's formulation of the concept for real numbers and Liouville’s 1844 construction of a transcendental number \cite{Liou}; Hermite did not show $e$ to be transcendental until 1873 \cite{Herm}. (For an historical overview of the history of transcendental numbers, see the first chapter of Baker \cite{Bak} or Section 22.2.3 of Suzuki \cite{Suz}.)

The following example follows Liouville's construction of a decimal expansion that must represent a transcendental number.

\begin{example}\label{Example: Irrationality in a Liouville style}
	Let $R = \bbC[z]$. Let $L(z) = \sum_{n \geq 0} z^{2^n}$; we note that the set of nonzero coefficients of $L^{p-1}$ is a proper subset of the set of nonzero coefficients of $L^p$. For $p<q$, define $c(p,q) \coloneqq 2^q(1-2^{-p})$; then we have $(L^p)_{c(p,q)} = q!$, $(L^j)_{c(p,q)} = 0$ for all $j<p$, and $(L^j)_n = 0$ for all $j \leq p$ and all $n$ with $0<|c(p,q)-n|<2^{q-p}$.

Suppose now that $A(T) \coloneqq \sum_{j \leq m} A_j T^j \in \bbC[z][T]$ is a degree $m$ polynomial. We select $n$ with $(A_m)_n \neq 0$, and let $d \coloneqq \max\{\deg(A_j):j \leq m\}$, where we adopt the convention $\deg (0) \coloneqq -\infty$. Then for any $q$ large enough that $2^{q-p} > d$, we have
\[
(A(L))_{c(p,q) + n} = (L^m)_{c(p,q)}\cdot(A_m)_n \neq 0.
\]
Thus $L$ is transcendental over $\bbC[z]$.
\end{example}

It is possible that the technique embodied in Example \ref{Example: Irrationality in a Liouville style} can be extended to construct a series $X$ transcendental over (for instance) the ring of absolutely convergent series. But the rate of growth of the coefficients of $X$, rather than of the gaps between them, may also be incompatible with the existence of a nonzero polynomial that takes $X$ to 0.

Recall that a series of the form $C(z) \coloneqq \sum_{n \gg -\infty} C_{n/d} z^{n/d}$ is a \defi{Puiseux series}, and 
\[
\bbC((z^*)) \coloneqq \bigcup_{d \geq 1} \bbC((z^{1/d}))
\]
is the \defi{field of Puiseux series}. If for some $r > 0$, $\abs{C_{n/d}} = O(r^{n/d})$ as $n \to \infty$, we say $C$ exhibits \defi{exponential growth}. Otherwise, $C$ exhibits \defi{superexponential growth}.

\begin{remark}
	A \emph{power series} $C(z) \in \bbC[[z]]$ exhibits exponential growth precisely if $C(z)$ converges for $z$ in a neighborhood of the origin. In other words, $C(z)$ exhibits exponential growth precisely if $C(z)$ defines an analytic function in a neighborhood of the origin.
\end{remark}

\begin{theorem}[Newton-Puiseux Theorem, \cite{Nowak}]\label{Theorem: Newton-Puiseux Theorem}
	The field $\bbC((z^*))$ is algebraically closed. Moreover, if a nonzero polynomial $A(T) \in \bbC((z^*))[T]$ has coefficients exhibiting exponential growth, then the roots of $A(T)$ exhibit exponential growth.
\end{theorem}

The following corollary is a straightforward application of the Newton-Puiseux Theorem (see \cite{Leason}).

\begin{cor}\label{Corollary: Superexponential series are transcendental}
	If $X \in \bbC[[z]]$ exhibits superexponential growth, then $X$ is transcendental over the ring of series exhibiting exponential growth.
\end{cor}

\begin{example} 
	The power series $\sum_{n \geq 0} n! z^n \in \bbC[[z]]$ is irrational over the ring of series exhibiting exponential growth. It is a fortiori irrational over the ring of Abel-summable series, the ring generated by the convergent series, and the ring of absolutely convergent series, since these are all subrings of the ring of series exhibiting exponential growth.
\end{example}

\begin{remark}
	The proof of the Newton-Puiseux Theorem applies without modification if we replace $\bbC$ with any characteristic 0 algebraically closed field equipped with an absolute value $\abs{\cdot}$. It may also be adapted to prove that fast-growing series are transcendental over other subrings of $\bbC[[z]]$ whose series exhibit modest growth.
\end{remark}

The object of this note is to establish by purely elementary means that if every series in $R$ has coefficients exhibiting at most modest growth, then any series $X \in \bbC[[z]]$ with sufficiently fast-growing coefficients is transcendental over $R$. We attain this result by means of an elementary algebraic identity (Theorem \ref{Theorem: Decomposition of the coefficients of A(X)}), rather than by invoking the machinery of Puiseux series. We state our main theorem (Theorem \ref{Theorem: Transcendence for fast-growing series}) here.

\begin{introtheorem}
	Fix $R$ a subring of $\bbC[[z]]$, and suppose we have a monotone increasing function $\rho : \bbN \to \bbR_{> 0}$ such that for every $C \in R$, we have $C_n = O(\rho(n))$ as $n \to \infty$.
	
	Let $X \in \bbC[[z]]$; suppose that $\abs{\xn 0} \geq 1$, and that for every fixed $\lambda$ and $m$, the power series $X$ satisfies the following conditions as $n \rightarrow \infty$: 
	\begin{align*}
		\rho(n) \parent{\sum_{\ell \leq \frac{n}{2}} \abs{\xn {\ell}}}^m &= o(\abs{\xn {n-\lambda}}); \\
		\rho(n) \abs{\xn {n - \lambda - 1}} \parent{\sum_{\ell < \frac{n}{2}} \abs{\xn \ell}}^m &= o(\abs{\xn {n-\lambda}}).
	\end{align*}
	Then $X$ is transcendental over $R$.
\end{introtheorem}

In particular, Theorem \ref{Theorem: Transcendence for fast-growing series} applies to the ring of power series exhibiting exponential growth, for instance when we take $\rho(n) \coloneqq n!$ (see Example \ref{Example: transcendental over series exhibiting exponential growth} below). Although the content of Theorem \ref{Theorem: Transcendence for fast-growing series} is implied by Corollary \ref{Corollary: Superexponential series are transcendental}, at least when $R$ is the ring of series exhibiting exponential growth, we find merit in our approach and hope that Theorem \ref{Theorem: Decomposition of the coefficients of A(X)} may find other applications in the future.

\begin{remark}
	Our work applies verbatim to power series over any characteristic 0 field $K$ equipped with an absolutely value $\abs{\cdot}$.
\end{remark}

The idea behind the proof of Theorem \ref{Theorem: Transcendence for fast-growing series} is fairly straightforward: we suppose we have a polynomial $A(T) \coloneqq \sum_{j \leq m} A_j T^j \in R[T]$ such that $A(X) = 0$, and deduce that $A(T) = 0$. To do so, we make a careful examination of $\axn n$ for $n$ large; if the coefficients of $X$ grow sufficiently rapidly, then the behavior of $\xn n$ will dominate $\axn n$ unless the coefficient of $\xn n$ in $\axn n$ is zero. But as $A(X) = 0$, the coefficient of $\xn n$ must be zero, and so the behavior of $\xn {n-1}$ will dominate $\axn n$ unless this coefficient is zero; proceeding inductively, for any $\ell$ small, we conclude the coefficient of $\xn {n-\ell}$ is zero. It turns out that for $\ell$ small relative to $n$, the coefficient of $\xn {n-\ell}$ in $\axn n$ is independent of $n$: in fact, the coefficient of $\xn {n - \ell}$ is precisely $\apxn \ell$, where $A\prm(X) \coloneqq \sum_{j\leq m}j A_j X^{j-1}$ is the formal derivative of $A(T)$ evaluated at $X$ (see Theorem \ref{Theorem: Decomposition of the coefficients of A(X)} below). Consequently, if $A(X) = 0$ we would expect to have $A\prm(X) = 0$, and this is indeed the case; now an easy bit of algebra tells us that $A(T) = 0$ as desired. In the remainder of this note, we formalize the intuition outlined above.

\section*{Acknowledgments}

We thank Gary Walsh and John Voight for their helpful observations.

\section{Algebraic techniques}

In this section, we give a relationship between the coefficients of $A(X)$ and $A\prm(X)$ which emphasizes the high-index coefficients of $X$.

\begin{lemma}\label{Lemma: Decomposition of x_n^m}
	Fix a power series $X \in \bbC[[z]]$. For any $m$, we have
	\begin{align}
		X^m &= X^{[m]} + m X^{\langle m \rangle}, \label{decomp}\\ 
\intertext{where the power series $X^{[m]}$ and $X^{\langle m \rangle}$ are defined by}
                \xmn{[m]}n & \coloneqq \sum_{\substack{k_1 + \ldots + k_m = n \\ k_1, \ldots, k_m \leq \frac {n}2}} \xn {k_1} \ldots \xn {k_m} \nonumber \\
\intertext{and}
    \xmn{\langle m \rangle} n
        &\coloneqq \sum_{\ell < \frac{n}{2}} \xmn {m-1}{\ell} \xn {n - \ell}. \nonumber
	\end{align}
\end{lemma}

\begin{proof}
       
	By definition, we have
\begin{align*}
	\xmn m n &= \sum_{k_1 + \ldots + k_m = n} \xn {k_1} \ldots \xn {k_m} \\
                 &= \sum_{\substack{k_1 + \ldots + k_m = n \\ k_1, \ldots, k_m \leq \frac {n}2}} \xn {k_1} \ldots \xn {k_m} + \sum_{j=1}^m \sum_{\substack{k_1 + \ldots + k_m = n \\ k_j > \frac {n}2}} \xn {k_1} \ldots \xn {k_m}.
\intertext{If $j = 1$ and $\ell \coloneqq n-k_1$, then $\ell < n/2$. By symmetry, the last sum is independent of $j$, so} 
        \xmn m n &= \xmn{[m]}n + m \sum_{\ell<n/2} \xn{n-\ell} \sum_{k_2 + \ldots + k_{m} = \ell} \xn {k_2} \ldots \xn {k_m}\\
                 &= \xmn{[m]}n + m \sum_{\ell<n/2} \xn{n-\ell} (X^{m-1})_\ell\;,
\end{align*}
from which \eqref{decomp} follows.
\end{proof}

Informally, Lemma \ref{Lemma: Decomposition of x_n^m} partitions the summands of $\xmn m n$ into a central ``core'' in which every index is less than or equal to $n/2$, and $m$ sets in which each summand has a (necessarily single) factor $\xn k z^k$ with $k>n/2$. We illustrate this with the summands of $\xmn 3 4$, the coefficient of $z^4$ in $X^3 = (\xn 0 + \xn 1 z + \xn 2 z^2 + \dots)^3$, arranged as
\begin{equation*}
\medmath{
\begin{matrix}
&&&&\mathbf{\xn 0 \xn 0 \xn 4}&&&&\\\
&&&\mathbf{\xn 0\xn 1\xn 3}&&\mathbf{\xn 1\xn 0\xn 3}&&&\\
&&\xn 0\xn 2\xn 2&&\xn 1\xn 1\xn 2&&\xn 2\xn 0\xn 2&&\\
&\mathbf{\xn 0\xn 3\xn 1}&&\xn 1\xn 2\xn 1&&\xn 2\xn 1\xn 1&&\mathbf{\xn 3\xn 0\xn 1}&\\
\mathbf{\xn 0\xn 4\xn 0}&&\mathbf{\xn 1\xn 3\xn 0}&&\xn 2\xn 2\xn 0&&\mathbf{\xn 3\xn 1\xn 0}&&\mathbf{\xn 4\xn 0\xn 0}.
\end{matrix}}
\end{equation*}
The terms in any of the three boldfaced triangles sum to $(X^{\langle 3 \rangle})_4$, while $(X^{[3]})_4$ is the sum over the central inverted triangle. 

These definitions extend to full power series. For instance, we have 
\[
X^{[3]}(z) = \xn 0 + 3\xn 0^2\xn 1 z + 3\xn 0\xn 1^2 z^2 + \xn 1^3 z^3 + (3\xn 0\xn 2^2+3\xn 1^2\xn 2)z^4 + \cdots
\]
and  
\[
X^{\langle 3 \rangle}(z) = \xn 0^2\xn 2 z^2 + (\xn 0^2\xn 3+2\xn 0\xn 1\xn 2) z^3 + (\xn 0^2\xn 4+2\xn 0\xn 1\xn 3)z^4 + \cdots.
\]
But observe that for $m = 0, 1, 2$ and $n$ arbitrary, the set
\[
\set{(k_1, \ldots, k_m) \in \bbZ^m \ : \ k_1 + \ldots + k_m = n \ \text{and} \ 0 \leq k_1, \ldots, k_m \leq n/2}
\]
may be empty or singleton, so there are a few irregularities for low powers: we have $X^{[0]} = 1$,  $X^{[1]} = \xn 0$, and  $X^{[2]} = 1 + \xn 1^2 z^2 + \xn 2^2 z^4+ \cdots$.

Let
\begin{equation*}
	\delta (A,X) \coloneqq \sum_{j \leq m} A_j X^{[j]},
\end{equation*}
and write $\delta_n(A, X) \coloneqq \delta(A, X)_n$. We may think of $\delta(A, X)$ as a ``pseudopolynomial'' version of $A(X)$ using only the cores of the powers $X^m$. Let also 
\begin{equation*}
	\eerrorn n(A,X) \coloneqq \sum_{j\leq m} \sum_{\substack{k+p+q=n \\ q<p\leq\frac n2}} j (A_j)_k \xmn {j-1} q \xn p
\end{equation*}
and, for $\lambda < n/2$, let
\begin{equation*}
	\gamma_{n,\lambda}(A,X) \coloneqq \sum_{\lambda \leq \ell < \frac{n}{2}} \apxn \ell \xn {n - \ell}.
\end{equation*}
While innocuous on its face, Lemma \ref{Lemma: Decomposition of x_n^m} is instrumental in the proof of the following algebraic identity. 

\begin{theorem}\label{Theorem: Decomposition of the coefficients of A(X)}
	Fix a power series $X \in \bbC[[z]]$ and a polynomial $A(T) \in \bbC[[z]][T]$. For any $n$ and any $\lambda< n/2$, we have
\begin{equation}
	\axn n = \sum_{\ell < \lambda} \apxn \ell \xn {n - \ell} 
                   + \gamma_{n,\lambda}(A,X) 
                   + \delta_n(A,X)
                   + \eerrorn n (A,X),\label{Equation: Decomposition for alpha_n}
\end{equation}
\end{theorem}

\begin{proof}
	Write $A(T) = \sum_{j \leq m} A_j T^j$, where $m$ is the degree of $A$. By definition,
	\begin{equation}\label{polydef}
		\axn n = \sum_{j \leq m} (A_jX^j)_n  = \sum_{j \leq m} \sum_{k \leq n} \amn jk \xmn j{n-k}.
	\end{equation}
	Substituting \eqref{decomp} into \eqref{polydef}, and recalling the definition of $X^{[j]}$, we obtain
	\begin{align*}
		\axn n &= \sum_{j \leq m} \sum_{k \leq n} \amn jk \parent{(X^{[j]})_{n-k} + j (X^{\langle j \rangle})_{n-k}}\\
		&= \sum_{j \leq m} \parent{(A_j X^{[j]})_n + j \sum_{k \leq n} \amn jk \sum_{q < \frac{n-k}{2}} \xmn {j-1}q \xn {n - k - q}}\\
		&= \delta_n(A,X) + \sum_{j \leq m} j \sum_{k \leq n} \amn jk \sum_{q < \frac{n-k}{2}} \xmn {j-1}q \xn {n - k - q}\\
	\end{align*}
	Let $p \coloneqq n-k-q$, so $q < (n-k)/2$ if and only if $q < p$. Thus
	\[
 		\axn n - \delta_n(A,X) = \sum_{\substack{k+p+q=n \\ q<p}} \parent{\sum_{j\leq m}j A_j}_k \xmn {j-1}q \xn p
	\]
	We break this sum into two cases, depending on whether  $p \leq n/2$ or $p > n/2$.

	In the first case, we get
	\[
		\sum_{j\leq m} \sum_{\substack{k+p+q=n \\ q<p\leq\frac n2}} j (A_j)_k \xmn {j-1}\ell \xn p = \eerrorn n(A,X).
	\]

	In the second case, $k$ and $q$ have no restriction other than that $k+p+q=n$, and
	\begin{align*}
		\sum_{\substack{k+p+q=n \\ \frac n2 <p\leq n}}\parent{\sum_{j\leq m}j A_j}_k \xmn {j-1}q \xn p &= \sum_{\frac n2 < p \leq n}\parent{\sum_{j\leq m} j A_j X^{j-1}}_{n-p}X_p \notag\\
		&= \sum_{\frac n2 < p \leq n} A'(X)_{n-p} X_p.
	\end{align*}
	Setting $\ell \coloneqq n-p$, \eqref{Equation: Decomposition for alpha_n} follows immediately.
\end{proof}

Figure \ref{barycentric}  may make what we have just done clearer. The coefficient $\axn n$ is the sum of a number of terms of the form $(A_j)_k X_{\ell_1}X_{\ell_2}\cdots X_{\ell_j}$, where $k+\ell_1+\cdots+\ell_j = n$.  If none of the indices $\ell_i$ exceeds $(n-k)/2$, the summand is part of the core. Otherwise, let $p$ be some maximal $\ell_i$, and let $q$ the sum of the others. For fixed $n$, the triplets $(k,p,q)$ may be considered as barycentric coordinates.

\begin{figure}
   \centering
    \includegraphics[width= 8cm]{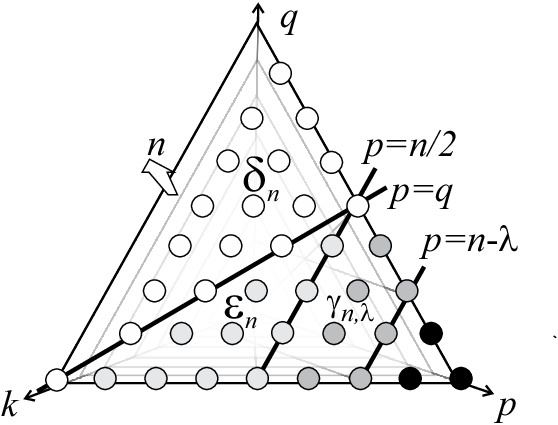}
      \caption{Components of $A(X)_n$ in barycentric coordinates}
    \label{barycentric}
\end{figure}

Triplets with $p\leq q$ (shown in white) are the summands of the core, $\delta_n(A,X)$. (Due to the definition of $p$, we cannot have $p=0$ and $q>0$ simultaneously; these positions on the graph are left empty.) Coordinates with $q < p \leq n/2$ (shown in light grey) are summands of $\eerrorn n (A,X)$. Finally, coordinates with $p>n/2$ correspond to the region in which the coefficient of $X_p$ in $A(X)_n$ is $A'(X)_{n-p}$. We subdivide this into the summands of $\gamma_{n,\lambda}(X,A)$ (dark grey) and the set of triplets (shown in black) for which $l < \lambda$. Unlike the others, this last set of triplets does not become more numerous with increasing $n$. We shall see that if the coefficients of $X$ grow fast enough, each of $\gamma_{n,\lambda}(X,A)$, $\delta_n(X,A)$, and $\epsilon_n(X,A)$ is insignificant compared to the terms with $l < \lambda$.

\section{Criteria and constructions for transcendental power series}\label{Section: Criteria and constructions for transcendental power series}

In this section, we prove our main theorem and furnish some related results.

\begin{lemma}\label{Lemma: If all three error terms are small, the derivative vanishes at X}
	Fix a power series $X \in \bbC[[z]]$ and a polynomial $A(T) \in \bbC[[z]][T]$, and suppose that $A(X) = 0$. Suppose moreover that for each $\lambda$, we have, as $n \rightarrow \infty$:
	\begin{align}
		\abs{\gamma_{n,\lambda}(A,X)} &= o(\abs{\xn {n-\lambda}}); \label{Condition: gamma_n is small} \\
		\abs{\derrorn n (A,X)} &= o(\abs{\xn {n-\lambda}});\label{Condition: delta_n is small} \\
		\abs{\eerrorn n (A,X)} &= o(\abs{\xn {n-\lambda}}).\label{Condition: epsilon_n is small}
	\end{align}
	Then $A\prm(X) = 0$.
\end{lemma}

\begin{proof}
	Let $n \geq 0$ be arbitrary. By Theorem \ref{Theorem: Decomposition of the coefficients of A(X)} we have
	\[
	\axn n = \sum_{\ell < \lambda} \apxn \ell \xn {n - \ell} 
                   + \gamma_{n,\lambda}(A,X) 
                   + \delta_n(A,X)
                   + \eerrorn n (A,X),
	\]
	We claim $\apxn \ell = 0$ for each $\ell$. If not, then let $\lambda$ be minimal such that $\apxn \lambda \neq 0$. Then
	\[
	\axn n = \apxn \lambda \xn {n - \lambda} 
                   + \gamma_{n,\lambda}(A,X) 
                   + \delta_n(A,X)
                   + \eerrorn n (A,X),
	\]
	By conditions \ref{Condition: gamma_n is small}, \ref{Condition: delta_n is small}, and \ref{Condition: epsilon_n is small}, we observe \[
	\abs{\gamma_{n,\lambda}(A,X)},\abs{\derrorn n (A,X)}, \abs{\eerrorn n (A,X)}  < \frac{\abs{\apxn \lambda \xn {n - \lambda}}}{3}
	\]
	for $n$ sufficiently large. Thus
	\[
	0 = \abs{\axn n} \geq \abs{\apxn \lambda \xn {n - \lambda}} - \abs{\derrorn n (A,X)} - \abs{\eerrorn n (A,X)} - \abs{\gamma_{n,\lambda}(A,X)} > 0,
	\]
	and we have obtained a contradiction. Then $\apxn \ell = 0$ for all $\ell$, and thus $A\prm(X) = 0$ as desired.
\end{proof}

\begin{lemma}\label{Lemma: If conditions 1-3 hold for all polynomials, then X is transcendental}
	Let $R$ be a subring of $\bbC[[z]]$ and let $X \in \bbC[[z]]$. Suppose that conditions \ref{Condition: gamma_n is small} through \ref{Condition: epsilon_n is small} of Lemma \ref{Lemma: If all three error terms are small, the derivative vanishes at X} hold for all series polynomials $A(T) \in R[T]$ such that $A(X) = 0$. Then $X$ is transcendental over $\bbC$.
\end{lemma}

\begin{proof}
	Suppose $A(T) \in R[T]$ is chosen with $A(X) = 0$. Repeated applications of Lemma \ref{Lemma: If all three error terms are small, the derivative vanishes at X} show that every derivative of $A(X)$ vanishes. Then $A(T)$ must be constant, and so $A(T) = A(X) = 0$. Thus $X$ is transcendental over $R$ as desired.
\end{proof}

At this point, we are ready to prove our main theorem.

\begin{theorem}\label{Theorem: Transcendence for fast-growing series}
	Fix $R$ a subring of $\bbC[[z]]$, and suppose we have a monotone increasing function $\rho : \bbN \to \bbR_{> 0}$ such that for every $C \in R$, we have $\abs{C_n} = O(\rho(n))$ as $n \to \infty$.
	
	Suppose that $\abs{\xn 0} \geq 1$, and that for every fixed $\lambda$ and $m$, the power series $X$ satisfies the following conditions as $n \rightarrow \infty$: 
	\begin{align*}
		\rho(n) \parent{\sum_{\ell \leq \frac{n}{2}} \abs{\xn {\ell}}}^m &= o(\abs{\xn {n-\lambda}}); \\
		\rho(n) \abs{\xn {n - \lambda - 1}} \parent{\sum_{\ell < \frac{n}{2}} \abs{\xn \ell}}^m &= o(\abs{\xn {n-\lambda}}).
	\end{align*}
	Then $X$ is transcendental over $R$.
\end{theorem}

\begin{proof}
	Note that the last condition of Theorem \ref{Theorem: Transcendence for fast-growing series} gives us $\abs{\xn n} = o(\abs{\xn {n+1}})$ as $n \to \infty$, and \emph{a fortiori} $\abs{X_n}$ is eventually increasing. We will prove that conditions \ref{Condition: gamma_n is small}, \ref{Condition: delta_n is small}, and \ref{Condition: epsilon_n is small} of Lemma \ref{Lemma: If all three error terms are small, the derivative vanishes at X} hold for every polynomial $A(T) \in R[T]$ with $A(X) = 0$, so Lemma \ref{Lemma: If conditions 1-3 hold for all polynomials, then X is transcendental} will give us our desired result. 

Fix a polynomial $A(T) = \sum_{j \leq m} A_j T^j \in R[T]$ for which $A(X) = 0$; and fix $\lambda \geq 0$. We compute (as $n \rightarrow \infty$): 
	\begin{align*}
		\abs{\gamma_{n,\lambda}(A,X)} &\leq \sum_{\lambda < \ell < \frac{n}{2}} \sum_{j \leq m} \sum_{k \leq \ell} \abs{j \amn jk \xmn {j - 1} {\ell - k} \xn {n - \ell} } \\
		&= O\parent{\rho(n) \sum_{\lambda < \ell < \frac{n}{2}} \sum_{k \leq \ell} \abs{\xmn {m - 1} {\ell - k} \xn {n - \ell} } } \\
		&= O\parent{n \rho(n) \abs{\xn {n - \lambda - 1} } \sum\limits_{k_1 + \dots + k_{m-1} < \frac{n}{2}} \abs{\xn {k_1} \dots \xn {k_{m-1}}}} \\
		&= O\parent{\rho(n) \abs{\xn {n - \lambda - 1} } \parent{\sum\limits_{\ell < \frac{n}{2}} \abs{\xn {\ell}}}^{m}} \\
		&= o(\abs{\xn {n - \lambda}}),
	\end{align*}	
	where the second-to-last asymptotic holds because $n = O\parent{\sum_{\ell < \frac{n}{2}} \abs{\xn \ell}}$. Thus condition \ref{Condition: gamma_n is small} holds.
	
	Next we compute
	\begin{align*}
		\abs{\derrorn n (A,X)} &\leq \sum_{j \leq m} \sum_{k \leq n} \abs{\amn jk (X^{[j]})_{n - k}} \\
		&= O\parent{\rho(n) \sum_{k \leq n} \sum_{\substack{k_1 + \ldots + k_m = n - k \\ k_1, \ldots, k_m \leq \frac {n-k}2}} \abs{\xn {k_1} \ldots \xn {k_m}}}\\
		&= O\parent{n \rho(n) \parent{\sum_{\ell \leq \frac {n}2} \abs{\xn {\ell}}}^m} \\
		&= O\parent{\rho(n) \parent{\sum_{\ell \leq \frac {n}2} \abs{\xn {\ell}}}^{m+1}} \\
		&= o(\abs{\xn {n - \lambda}}),
	\end{align*}
	where the second-to-last asymptotic holds because $n = O\parent{\sum_{\ell \leq \frac{n}{2}} \abs{\xn \ell}}$. Thus condition \ref{Condition: delta_n is small} holds.
	
	Finally, we compute
	\begin{align*}
		\abs{\eerrorn n (A,X)} &\leq \sum_{j\leq m} \sum_{\substack{k+p+q=n \\ q<p\leq\frac n2}} \abs{j (A_j)_k \xmn {j-1} q \xn p} \\
		&= O\parent{n \rho(n) \abs{\xn {\floor{\frac{n}{2}}}} \sum_{\substack{k + q < n \\ q < \frac n2}} \abs{\xmn {m-1} q}} \\
		&= O\parent{n^2 \rho(n) \abs{\xn {\floor{\frac{n}{2}}}} \parent{\sum_{q < \frac n2} \abs{\xn q}}^{m-1} } \\
		&= O\parent{\rho(n) \parent{\sum_{\ell \leq \frac n2} \abs{\xn \ell}}^{m+2}} \\
		&= o(\abs{\xn {n - \lambda}}),
	\end{align*}
	where the second-to-last asymptotic holds because $n$, $\sum_{q < \frac{n}{2}} \abs{\xn q}$, and $\abs{\xn {\floor{\frac{n}{2}}}}$ are each $O\parent{\sum_{\ell < \frac{n}{2}} \abs{\xn \ell}}$. Thus condition \ref{Condition: epsilon_n is small} holds, and $X$ is transcendental over $R$.
\end{proof}

\begin{corollary}\label{Corollary: if X satisfies nice-ish conditions it is transcendental}
	Assume the notation of Theorem \ref{Theorem: Transcendence for fast-growing series}. Suppose that $\abs{\xn 0} \geq 1$, and that for every fixed $\lambda$ and $m$, the power series $X$ satisfies the following condition as $n \to \infty$:
	\[
		\rho(n) \abs{\xn {n - \lambda - 1}} \parent{\sum_{\ell \leq \frac{n}{2}} \abs{\xn {\ell}}}^m = o(\abs{\xn {n-\lambda}}).
	\]
	Then $X$ is transcendental over $R$.
\end{corollary}

\begin{ex}\label{Example: transcendental over series exhibiting exponential growth}
	Let $R \subseteq \bbC[[z]]$ comprise the power series which exhibit exponential growth. Then we may take $\rho(n) = n!$, since $n!$ exhibits superexponential growth. Let $X = \sum_{n \geq 0} 2^{n!} z^n$, so $\xn n = 2^{n!}$. Then $\abs{\xn 0} = 2 > 1$. Moreover, for every fixed $\lambda$ and $m$, we have (as $n \rightarrow \infty$):
	\[
		n! \parent{\sum_{\ell \leq \frac{n}{2}} 2^{\ell!}}^m \leq n! \parent{\sum_{\ell \leq \floor{\frac{n}{2}}!} 2^{\ell}}^m = O\parent{2^{\floor{\frac{n+4}{2}}!}} = o\parent{2^{\parent{n - \lambda}!}},
	\]
so $X$ is transcendental over $R$ by Corollary \ref{Corollary: if X satisfies nice-ish conditions it is transcendental}. The series $X$ is a fortiori transcendental over the ring of Abel-summable series, the ring generated by the convergent series, and the ring of absolutely convergent series, since these are all subrings of $R$.
\end{ex}

\end{document}